\documentclass[12pt,a4paper]{amsart}%

\usepackage{imakeidx}
\usepackage{amsmath}
\usepackage{fullpage}
\usepackage{comment}
\usepackage{MyCommA}
\usepackage[final]{pdfpages}
\usepackage{tikz-cd}

\makeindex

\newcommand{\Rami}[1]{{{#1}}}
\newcommand{\RamiA}[1]{{{#1}}}
\newcommand{\Eitan}[1]{{{#1}}}

\newcommand{\Nir}[1]{{{#1}}}

\newcommand{\nircommentsout}[1]{}

\newcommand{\NextVer}[1]{}

\newcommand{\noleft}{\left.\kern-\nulldelimiterspace}

\DeclareMathOperator{\Fl}{Fl}
\DeclareMathOperator{\linspan}{span}
\DeclareMathOperator{\divv}{div}

\begin{document}
	
\author[Aizenbud]{Avraham Aizenbud}
\address{Avraham Aizenbud,
Faculty of Mathematical Sciences,
Weizmann Institute of Science,
76100 Rehovot, Israel}
\email{aizenr@gmail.com}
\urladdr{https://www.wisdom.weizmann.ac.il/~aizenr/}

\author[Avni]{Nir Avni}
\address{Nir Avni,
Department of Mathematics, 
Northwestern University,
2033 N. Sheridan Rd,
Evanston IL, 60201}
\email{avni.nir@gmail.com}

\author[Gourevitch]{Dmitry Gourevitch}
\address{Dmitry Gourevitch,
Faculty of Mathematical Sciences,
Weizmann Institute of Science,
76100 Rehovot, Israel}
\email{dimagur@weizmann.ac.il}
\urladdr{https://www.wisdom.weizmann.ac.il/~dimagur/}

\author[Kazhdan]{David Kazhdan}
\address{David Kazhdan,
Einstein Institute of Mathematics,
Edmond J. Safra Campus, Givaat Ram,
The Hebrew University of Jerusalem,
Jerusalem 91904, Israel}
\email{david.kazhdan@mail.huji.ac.il}
\urladdr{https://math.huji.ac.il/~kazhdan/}

\author[Sayag]{Eitan Sayag}
\address{Eitan Sayag,
Department of Mathematics,
Ben Gurion University of the Negev,
P.O.B. 653,
Be'er Sheva 84105, Israel}
\email{eitan.sayag@gmail.com}
\urladdr{https://www.math.bgu.ac.il/~sayage}

\date{\today}


\title{A proof of Harish-Chandra's integrability theorem for cuspidal representations of $\GL_n(\F_\ell((t)))$}

\maketitle
\begin{abstract}
\RamiA{Consider} the Chevalley map
$$
p:\fg\fl_n(F)\to (\fg\fl_n//\GL_n)(F),
$$
where $F=\F_\ell((t))$.
\RamiA{W}e show that \RamiA{the push forward via $p$ of} every smooth compactly supported measure on $\fg\fl_n(F)$ \RamiA{is} a measure whose density belongs to $L^q$ for every finite $q$. 

As a consequence, using the main result of \cite{AGKS2}, we obtain 
local integrability 
for Harish--Chandra's characters of irreducible cuspidal representations of $\GL_n(F)$.
\end{abstract}

\tableofcontents

\section{Introduction}
\subsection{Main results}

Let $\ell$ be a prime power and $n$ be an integer. Let $F=\F_\ell((t))$, and let $\ug=\fg\fl_n$.
Let $\uc:=\ug//\GL_n$ be the categorical quotient. Let $\g=\ug(F)$ and $\fc:=\uc(F)$.
Let $p:\ug\to\uc$ be the Chevalley map. \index{Chevalley map}
Concretely, identifying $\uc$ with $\mathbb A^n$, the map $p$ is given by
{
$$
p(A)=(p_1(A),\dots,p_n(A)),
$$
where
$$
\det(\bfx I-A)
=
\bfx^n+p_1(A)\bfx^{n-1}+\cdots+p_n(A).
$$
}

\RamiA{Given a smooth compactly supported measure $\mu$ on $\g$, the measure $p_*(\mu)$ is absolutely continuous with respect to a Haar measure $\mu_{\fc}$\index{$\mu_{\fc}$} on $\fc$. Thus the 
Radon-Nikodim derivative $\frac{p_*(\mu)}{\mu_{\fc}}$ exists.}

In this paper, we prove the following:

\begin{introtheorem}[\S\ref{sec:Pf.alm.an.frs}]\label{thm:intro-main}
For any smooth compactly supported measure $\mu$ on $\g$, \RamiA{we have }
$$
\frac{p_*(\mu)}{\mu_{\fc}}\in \bigcap_{1\le q<\infty}L^q(\fc)
$$
where $\mu_{\fc}$ is a Haar measure on $\fc$.
\end{introtheorem}

By \cite[\S13(5)]{AGKS2}, \Cref{thm:intro-main} implies the following:

\begin{introtheorem}\label{thm:B}
Let $F=\F_\ell((t))$.
Let $\pi$ be a \RamiA{(complex)} irreducible cuspidal smooth representation of $\GL_n(F)$. Then the Harish--Chandra's character of $\pi$ is represented by a locally integrable function on $\GL_n(F)$.
\end{introtheorem}
\subsection{Background}
Harish--Chandra proved the analogous theorems {for every reductive group ${\bf G}$} in characteristic zero (\cite{HC_VD}). {In \cite{HC_Q} (see also \cite{HC_int}) he showed that the local integrability of characters result holds for any irreducible representation.}

Throughout the years, analogs of Harish-Chandra's integrability theorem in positive characteristic were proven in some generality (See \cite{JL,Rod,CGH,AGKS2,Tsa}). For a detailed overview on the history of this problem before \cite{AGKS2} we refer the reader to \cite[\S 1.2.1]{AGKS2}. In \cite{AGKS2} we proved \Cref{thm:B} under the additional assumption of existence of resolution of singularities or under the assumption $\RamiA{\chara(F)}>n/2$ (see \cite[Theorem C and Proposition D,  and \S 12]{AGKS2}).  Since then an analog of Harish-Chandra's integrability theorem in positive characteristic was proven in \cite{Tsa} under some assumptions \RamiA{including the requirement} that  \RamiA{$\chara(F)$} does not divide the order of $W_{\bf G}$, the Weyl group of $\bf G$.

In \cite[Theorem D]{AGKS_jets} we proved \Cref{thm:intro-main} under the additional assumption of existence of resolution or the assumption $\RamiA{\chara(F)}>n/2$ (see \cite[Theorems D and E]{AGKS_jets}). 
We refer the reader to \cite[\S 1.3]{AGKS_jets} for a discussion of the background of this problem.
The approach of the current paper is different from the approach of  \cite{AGKS_jets}  and does not depend on any additional assumptions.

\subsection{Idea of the proof}
\RamiA{To prove \Cref{thm:intro-main} we reduce it to integrability of a specific function $|\cJ(A)|^{-s}$ on $\g$. We prove this integrability using an explicit embedded resolution of the $0$-locus of $\cJ$.}

\RamiA{In more details, t}he proof of \Cref{thm:intro-main} has three main steps.

\textbf{Step 1: Reduction to the Log-canonicity of the partial Jacobian of the Chevalley map.}

In \Cref{lem:last-column-affine} we observe that the Chevalley map in the case at hand is partially affine linear. Specifically, after decomposing a matrix $A\in\g$ into its first $n-1$ columns, parametrized by $Y=\operatorname{Mat}_{n\times(n-1)}(F)$, and its last column $v\in V=F^n$, the Chevalley map $p:\g\to\fc$ is affine linear in $v$ with fixed $y\in Y$, i.e. it has the form $p(y,v)=\cD(y)v+g(y)$, where $\cD(y)\in\operatorname{End}(\fc)$ is the partial Jacobian matrix of $p$. 
We then establish in \Cref{thm:Lq-criterion} a general criterion that controls the $L^q$ norm of the pushforward \RamiA{of smooth compactly supported measures} under such maps in terms of the integral of $|\det(\cD)|^{-s}$ over $Y$ for every $s<1$. 

\textbf{Step 2: Explicit formula for the Jacobian.}

The function $\cJ(A):=\det(\cD(A))$ is computed explicitly in \Cref{prop:Jacobian of p}: up to a  change of coordinates it is $\det(v, Av,\dots,A^{n-1}v)$ for some fixed nonzero vector $v\in F^n$. 

\textbf{Step 3: Log-canonicity of the Jacobian.}
By Step 1, we reduce to the statement that $|\cJ|^{-s}$ is locally integrable over $Y$ for every $s<1$. It is enough to show this local integrability over $\g$. 
We show this local integrability in \Cref{thm:logcanonical}. For this, we use the classical method \RamiA{of Bernstein-Gelfand and Atiyah} (See \cite{BG,Atiyah}). This method relies on resolution of singularities, which, in general, is not available in positive characteristic. However, in the case at hand, we found an explicit resolution that enables the \RamiA{implementation} of  \RamiA{that} method.

The key geometric ingredient is the construction of an embedded resolution  $\pi:\mathbf X\to\ug$ of the divisor given by $\cJ$.  Specifically, $\mathbf X$ parametrizes pairs $(\mathcal F,A)$ of a full flag  \RamiA{$$\mathcal F:=\left(
\{0\}=F_0\subset F_1 \subsetneq\cdots\subsetneq F_n=V\right)
,$$}
and a matrix $A\in \ug$ such that:
\begin{enumerate}
    \item 
    \RamiA{$F_1= \operatorname{span}(v)$}
where $v$ is the vector from Step 2.
\item 
The matrix $A\in\ug$ satisfies $AF_i\subseteq F_{i+1}$ for all $1\le i\le n-1$.    
\end{enumerate}
This construction is described and studied in \S\ref{sec:logcanonical}.
In local coordinates on $\mathbf X$, the pullbacks of $\cJ$ and the Jacobian of $\pi$ both become monomials.  \Cref{lem:modification} provides a description of the exponents of these monomials.
The local integrability 
of $|\cJ|^{-s}$
is then reduces to an elementary inequality on the exponents occurring in these monomials.

\subsection{Outlook}
We expect that the main results of \cite{AGKS2} \RamiA{(and in particular the fact that \Cref{thm:intro-main}  implies \Cref{thm:B})} are also valid for characters of \RamiA{all (not necessarily cuspidal)} irreducible representations. This is \RamiA{the} goal of our work in progress \cite{AGKS_eff,AGKS_int}. This would imply that \Cref{thm:B}  is also valid for  characters of \RamiA{all} irreducible representations.

Additionally, 
\RamiA{the work in progress \cite{DKS} is aimed at generalizing the results of \cite{AGKS2} to a wider class of groups including, at least, all groups of types B and C in odd characteristic.} 

The proof of \Cref{thm:intro-main} in the current paper heavily relies on the assumption $\ug= \fg\fl_n$. \RamiA{However, we have a strategy to generalize it to many other groups including all classical groups in odd characteristic. This is the goal of our work in progress \cite{AAGKS_Class}. Provided a successful completion of the above mentioned works, this will extend \Cref{thm:B} to classical groups of types B,C in odd characteristic.}

\subsection{Structure of the paper}
In \S \ref{Sec.Conv.} we fix our notation. 

In \S \ref{sec:Lq-criterion} we study the push of smooth measures under \Rami{partially affine linear maps}. We provide a simple criterion for such a pushed measure to be in $L^{q}$, namely \Cref{thm:Lq-criterion}. This is \Rami{the main} part of  Step 1 above.

In \S \ref{sec:J} we carry out Step 2 \Rami{and prove \Cref{prop:Jacobian of p}}.

In \S \ref{sec:logcanonical}
we carry out Step 3: we show that $\cJ$ is integrable by using an explicit \Rami{embedded} resolution of singularities and \Rami{by} controlling the relevant divisors. The key statement \Rami{here} is \Cref{lem:modification}. 

Finally, in \S \ref{sec:Pf.alm.an.frs} we deduce \Cref{thm:intro-main}.

\subsection{Acknowledgments}
We are deeply grateful to Mircea Mustaţă for a very useful e-mail correspondence.


During the preparation of this paper, A.A., D.G. and E.S. were partially supported by the ISF grant no. 1781/23, A.A. and N.A. were partially supported by BSF grant no. 2022193, N.A. was partially supported by NSF grant no. 2503233, and D.K. was partially supported by an ERC grant 101142781.

\section{Conventions}\label{Sec.Conv.}

Throughout the paper:

\begin{itemize}

\item $\ell$ is a prime power, $F:=\F_\ell((t))$, $O_F:=\F_\ell[[t]]$.
\index{$F$}
\index{$O_F$}
\index{$\ell$}
\item all algebraic varieties that we consider are defined over $F$.
\item Algebraic varieties are usually denoted by bold letters, while their sets of $F$-points are denoted by the corresponding usual letters. For example, $\mathbf X$ denotes an algebraic variety and $X=\mathbf X(F)$.

\item for Gothic letters we will usually use underlined letters instead of boldface.
\item We denote by $n$ a positive integer. \index{$n$}
\item $\ug:=\fg\fl_n$ and $\g:=\ug(F)$.
\index{$\ug$}
\index{$\g$}
\item We write $\uc=\ug//\GL_n$ for the categorical quotient and $\fc=\uc(F)$. We identify $\uc$ with the collection of monic polynomials of degree $n$ and with the affine space $\bA^n$. 
\index{$\uc$}
\index{$\fc$}
\item The Chevalley map is induced by the inclusion $F[\ug]^{\GL_{n}} \subset F[\ug]$ and is denoted by $p:\ug\to\uc$. We will treat it using the explicit formula given in the introduction:
$$
p(A)=(p_1(A),\dots,p_n(A)),
$$
where \index{$p$}
$$
\det(\bfx I-A)
=
\bfx^n+p_1(A)\bfx^{n-1}+\cdots+p_n(A).
$$
\item We use the same notation for algebraic maps and the induced maps on $F$-points.
\item A modification of an algebraic variety $\bfX$ is a proper birational map $p:\tilde \bfX\to \bfX$. 
\index{modification}
\item Depending on the context we will consider a divisor either as a subvariety or as a $\Z$-combination of irreducible varieties.
\index{Divisor}
\item For a function $f$ (or a section of a line bundle) on an algebraic variety $\bfX$ we denote by $\div(f)$ the corresponding divisor.
\index{$\div$}
\item An SNC divisor on an algebraic variety $\bfX$ is a strict normal crossings divisor (see \cite[\href{https://stacks.math.columbia.edu/tag/0BI9}{Definition 0BI9}]{SP}). 
\index{SNC divisor}

\item For an algebraic variety $\mathbf X$ over $F$, we consider the collection of its $F$-points as an $l$-space (in the sense of \cite{BZ}). If the variety is smooth we also consider this collection as an $F$-analytic  manifold.
\item For an $l$-space $X$,
we denote by $C_c^\infty(X)$ the space of locally constant compactly supported functions on $X$ and by $C^\infty(X)$ the space of locally constant functions on $X$.
\index{$C_c^\infty(X)$}
\item \RamiA{For an open set $U\subset F^d$, a  smooth measure on $U$ is a measure that can be written as $f\mu$ where $\mu$ is the Haar measure on $F^d$ and $f\in C^\infty(U)$.} For an $F$-analytic manifold $X$, a smooth measure on $X$ is a measure that 
\RamiA{is smooth on every local chart.}
\item For a \Nir{continuous function $\phi:X\to Y$ of topological spaces} and a compactly supported \RamiA{(Radon)} measure $\mu$ on $X$, we denote by $\phi_*(\mu)$ the pushforward measure on $Y$.
\index{$\phi_*(\mu)$}
\item If $\mu,\nu$ are measures on a space $X$, we denote by $\frac{\mu}{\nu}$ the corresponding Radon-Nikodim derivative. 
\item By an $L_{\RamiA{loc}}^q$ function on an $F$-analytic manifold we mean function which is $L_{\RamiA{loc}}^q$ w.r.t. (\RamiA{every}) smooth measure with full support. 
\end{itemize}

\section{Criterion for the $L^q$ property of pushforwards}\label{sec:Lq-criterion}

In this section we prove a general criterion that gives $L^q$ bounds for pushforwards of measures under maps which are affine linear in some of the variables.

\begin{theorem}\label{thm:Lq-criterion}
Let \RamiA{$\bfX,\bfV$} be finite-dimensional $F$-vector spaces, \RamiA{considered as algebraic varieties}, and let
$$
f:\RamiA{\bfX}\to \operatorname{End}(\RamiA{\bfV}),
\qquad
g:\RamiA{\bfX}\to \RamiA{\bfV}
$$
be \RamiA{morphisms of algebraic varieties}.
Define
$$
\phi:\RamiA{\bfX}\times \RamiA{\bfV}\to \RamiA{\bfV}
\quad \text{by} \quad
\phi(x,v)=f(x)v+g(x).
$$
\RamiA{Let $X:=\bfX(F)$ and let $V:=\bfV(F)$.}
Let $\mu_X,\mu_V$ be Haar measures on $X$ and $V$, respectively, and let $s\in(0,1)$.

Assume that for every $\sigma\in C_c^\infty(X)$, we have
$$
\int_X \sigma(x)|\det(f(x))|^{-s}\mu_X<\infty.
$$
Then for every $\rho\in C_c^\infty(X\times V)$, we have
$$
\frac{\phi_*(\rho\cdot (\mu_X\times\mu_V))}{\mu_V}
\in
L^{\frac1{1-s}}(V).
$$
\end{theorem}

\begin{proof}[Idea of the proof]
For every fixed $x\in X$, the map $\phi_x(v)=f(x)v+g(x)$ is affine linear. When $f(x)$ is invertible, the pushforward of a smooth measure under $\phi_x$ has density whose $L^q$ norm is controlled by
$$
|\det(f(x))|^{-s},
\qquad
s=1-\frac1q.
$$

Integrating over $x$ and applying Minkowski's inequality reduces the theorem to the assumed integrability condition.
\end{proof}

\begin{proof}
Fix $\rho\in C_c^\infty(X\times V)$ and let $\nu:=\phi_*(\rho\cdot (\mu_X\times\mu_V))$. Then $\nu=\int_{{X}} \nu_x\,\mu_X$, where $$\nu_x:=(\phi_x)_*(\rho(x,\cdot)\mu_V).$$ Here we used the notation $\phi_x$ for the affine linear function $\phi(x,\cdot)$. More explicitly, for every $x\in X$, we define
$$
\phi_x:V\to V,
\qquad
\phi_x(v):=\phi(x,v)=f(x)v+g(x).
$$

Clearly, $$
\frac{\nu}{\mu_V}(w)
=
\int_X h_x(w)\mu_X.
$$
where $h_x:=\frac{\nu_x}{\mu_V}$.

By Minkowski's integral inequality,
$$
\left\|\frac{\nu}{\mu_V}\right\|_{L^q(V)}
\le
\int_X \|h_x\|_{L^q(V)}\mu_X.
$$

Let $U\subset X$ be the non-vanishing locus of $\det(f(x))$. Note that by the
integrability assumption
its complement is of zero measure.
Now for $x \in U$, the linear map $f(x)$ is invertible and in this case,
$$
h_x(w)
=
\rho(x,f(x)^{-1}(w-g(x)))
|\det(f(x))|^{-1}.
$$

Therefore, denoting $q:=\frac1{1-s}$, for any $x \in U$ we have:
$$
\|h_x\|_{L^q(V)}^q
=
\int_V
|\rho(x,f(x)^{-1}(w-g(x)))|^q
|\det(f(x))|^{-q}
\mu_V
=
|\det(f(x))|^{1-q}
\int_V |\rho(x,u)|^q\mu_V.
$$

Since $1-q=-sq$, it follows that
$$
\|h_x\|_{L^q(V)}
=
|\det(f(x))|^{-s}
\|\rho(x,\cdot)\|_{L^q(V)}.
$$

Therefore
$$
\left\|\frac{\nu}{\mu_V}\right\|_{L^q(V)}
\le
\int_X
|\det(f(x))|^{-s}
\|\rho(x,\cdot)\|_{L^q(V)}
\mu_X.
$$

The function $x\mapsto \|\rho(x,\cdot)\|_{L^q(V)}$ belongs to $C_c^\infty(X)$. Hence, by assumption,
$$
\int_X
|\det(f(x))|^{-s}
\|\rho(x,\cdot)\|_{L^q(V)}
\mu_X
<
\infty.
$$

Thus
$$
\left\|\frac{\nu}{\mu_V}\right\|_{L^q(V)}
<
\infty,
$$
which proves that
$$
\frac{\phi_*(\rho\cdot (\mu_X\times\mu_V))}{\mu_V}
\in
L^{\frac1{1-s}}(V).
$$
\end{proof}

\section{The Jacobian of the Chevalley map with respect to a column: the function $\cJ$}\label{sec:J}

In this section, we find a formula for the Jacobian of the restriction of the Chevalley map to the last column.

We begin with the following simple observation.

\begin{lemma}\label{lem:last-column-affine}
The map $p:\ug\to\uc$ is affine linear when considered as a function of the last column, while all other columns are fixed.
\end{lemma}

\begin{proof} 
The coordinates of $p$ are the coefficients of the characteristic polynomial
$
\det(\mathbf{x}I-A).
$
Each coefficient is a combination of minors. Each minor that involves the last column depends on it linearly. The others are constant in this column. This proves the assertion.
\end{proof}

In the following, $V=F^n$ and ${\bf e}_1,\ldots,{\bf e}_n$ is the standard basis of $V$.
\begin{definition} \label{def:D and J} For $A\in \mathfrak{g}$, define a map $\mathcal{D}(A):V \rightarrow \mathfrak{c}$\index{$\cD$} by
\[
\mathcal{D}(A)(v)=\left. \frac{d}{d\RamiA{s}}\right|_{\RamiA{s}=0} p(A+\RamiA{s}v{\bf e}_n^T)
\]
and let $\mathcal{J}(A)=\det(\mathcal{D}(A))$.\index{$\cJ$}
\end{definition} 

The main result of this section is the following:

\begin{proposition} \label{prop:Jacobian of p} $\mathcal{J}(A)=\RamiA{(-1)^n}\det({\bf e}_n^T,{\bf e}_n^TA,\ldots,{\bf e}_n^TA^{n-1})$.
\end{proposition}

For a variable $\bf{x}$, the adjugate matrix $\adj({\bf{x}}I-A)$ is an $n$-by-$n$ matrix whose entries are polynomials of degree at most $n-1$ in $\bf{x}$. Define $C_0(A),\ldots,C_{n-1}(A)\in \End(V)$ by $\adj({\bf{x}}I-A)=\sum_{\RamiA{k=0}}^\RamiA{n-1} C_k(A){\bf{x}}^k$.\index{$C_k$} 

\begin{lemma} \label{lem:adjugate} $ $ \begin{enumerate}
\item $C_{n-1}(A)=I$.
\item For every $k<n-1$, $C_k(A)=A^{n-k-1}+\text{linear combination of $C_{k+1}(A),\ldots,C_{n-1}(A)$}$.
\end{enumerate} 
\end{lemma} 

\begin{proof} Both claims follow from $({\bf{x}}I-A) \cdot \adj({\bf{x}}I-A)=\det({\bf{x}}I-A) \cdot I$.
\end{proof}

\begin{proof}[Proof of \Cref{prop:Jacobian of p}] Denote the vector space of polynomials in $\bf{x}$ of degree at most $n-1$ by $W$. The elements ${\bf 1},{\bf x},\ldots,{\bf x}^{n-1}$ form a basis of $W$. \Rami{By \Cref{lem:last-column-affine}} the function $v\mapsto p(A+\RamiA{v}{\bf e}_n^T)=\det({\bf x}I-A\RamiA{-}v{\bf e}_n^T)$ is affine \Rami{linear}, so 
\[
\mathcal{D}(A)(v)=\det \left( {\bf x}I-A\RamiA{-}v \bfe_n^T\right) - p(A)({\bf x})
\]
and $\mathcal{D}(A)({\bf e}_i)= \RamiA{-}\adj({\bf x}I-A)_{n,i}$. In the bases $\left\{ {\bf e}_i \right\}$ and $\left\{ {\bf x}^i \right\}$, the $(i,j)$ entry of the matrix representing $\mathcal{D}(A)$ equals  $\RamiA{-}C_{i}(A)_{n,j}$. This matrix is the matrix whose rows are $\RamiA{-}\bfe_n^TC_i(A)$. By Lemma \ref{lem:adjugate}, the determinant of $\mathcal{D}(A)$ is equal to the determinant of the matrix whose rows are $\RamiA{-\bfe_n^T,-\bfe_n^TA,\ldots, -\bfe_n^TA^{n-1}}$.

\end{proof}

\section{Log-canonicity of $\cJ$}\label{sec:logcanonical}

In this section we prove the following theorem.

\begin{theorem}\label{thm:logcanonical}
Let $s<1$ and let $\rho\in C_c^\infty(\g)$. Then
$$
\int_{\g}
|\cJ(x)|^{-s}\rho(x)\,dx
<
\infty,
$$
where $dx$ is a Haar measure on $\g$.
\end{theorem}

\begin{proof}[Idea of the proof]
By \Cref{prop:Jacobian of p}, it suffices to prove the local integrability of $|g|^{-s}$ where $g(A)=\det({\bf e}_1,A{\bf e}_1,\dots,A^{n-1}{\bf e}_1)$. The key ingredient is a modification $\pi:\mathbf X\to\ug$ constructed in \Cref{lem:modification}, where $\mathbf X$ parametrizes pairs $(\mathcal F, A)$ where $\cF$ is a full flag containing $\linspan(v)$ and $A$ is a  matrix of degree $\leq 1$ with respect to this flag. As the map $\pi$ is a modification, it suffices to check the integrability after pulling back to $\mathbf X(F)$. On $\mathbf X$, \Rami{the divisors of} both $\pi^*g$ and the Jacobian of $\pi$ \Rami{are SNC} divisors with explicit exponents: the $i$-th irreducible component $\mathbf D_i$ appears with multiplicity $i+1$ in $\operatorname{div}(\pi^*g)$ and multiplicity $i$ in $\operatorname{div}(\operatorname{Jac}(\pi))$. Working locally in the analytic topology near any point of $\mathbf X(F)$, the integrability of the pullback reduces to the \Rami{local integrability of $|x|^{-s}$} when $s<1$.
\end{proof}

In order to prove \Cref{thm:logcanonical} we construct an explicit monomialization of the function $g\in \mathcal{O}_{\ug}$ given by
\[
g(A)=\det({\bf e}_1,A{\bf e}_1,\ldots,A^{n-1}{\bf e}_1).
\]

\begin{definition} 
$ $
\begin{enumerate}
\item
$\Fl(V)$\index{$\Fl(V)$} is the flag variety of $V$, i.e., the variety parameterizing full flags
\[
0=F_0 \subseteq F_1 \subseteq \cdots \subseteq F_n=V,
\]
where $F_i \subseteq V$ is an $i$-dimensional subspace of $V$. 
\item $\Fl(V)_1 \subseteq \Fl(V)$ is the subvariety parameterizing flags $$F_0 \subseteq F_1 \subseteq \cdots \subseteq F_n$$ such that $F_1=\linspan({\bf e}_1).$\index{$\Fl(V)_1$}
\item $\mathbf{X} \subseteq \Fl(V)_1 \times \Rami{\ug}$ is the variety parameterizing pairs $((F_i),A)$ such that $AF_i \subseteq F_{i+1}$.\index{$\mathbf{X}$}
\end{enumerate}
\end{definition}

$\Fl(V)_1$ is isomorphic to the flag variety of $V/\linspan({\bf e}_1)$ and the projection $\phi : \mathbf{X} \rightarrow \Fl(V)_1$ is a vector bundle. In particular, $\mathbf{X}$ is irreducible and smooth. 

We now construct an affine atlas for $\bfX$.
\begin{notation}
Denote
\begin{enumerate}
    \item 
$\bfP:=\Stab_{\GL(V)}(\linspan({\bf e}_1))$. It acts on $\mathbf{X}$ by $g \cdot ((F_i),A)=((gF_i),g A g ^{-1})$.\index{$\bfP$}
\item 
$\bfN$ is the intersection of $\bfP$ and the group of unipotent lower-triangular matrices.  
\item 
$\mathcal{F}_{id}$\index{$\cF_{id}$} the standard flag 
\[
0 \subseteq \linspan({\bf e}_1) \subseteq \linspan({\bf e}_1,{\bf e}_2) \subseteq \cdots \subseteq \linspan({\bf e}_1,\ldots,{\bf e}_{n-1}) \subseteq V.
\]
\item 
$
\RamiA{\uv}:= \left\{ A\in \Rami{\ug} \mid A_{i,j}=0 \text{ if $i>j+1$}\right\}.
$
Note that 
 the fiber $\phi^{-1}(\mathcal{F}_{id})$ is identified with $\RamiA{\uv}$.\index{$\uv$}
\item
 $\varphi_{id}:\bfN \times \RamiA{\uv} \rightarrow \mathbf{X}$ is the map given by $\Rami{(u,A) \mapsto u \cdot (\mathcal{F}_{id},A)}.$\index{$\varphi_{id}$}
 \item 
 For every \Rami{$\RamiA{h} \in P:=\bfP(F)$, set
  $\varphi_{\RamiA{h}}:\RamiA{\bfN}\times \RamiA{\uv} \rightarrow \mathbf{X}$ to be the map defined by $\varphi_{\RamiA{h}}(u,A):=\RamiA{h}\cdot \varphi_{id}(u,A)$}.\index{$\varphi_{h}$} 
\end{enumerate}

\end{notation}
\Rami{
The following lemma is standard:
\begin{lemma}
    $ $
    \begin{itemize}
        \item For any $\RamiA{h}\in P$, the map $\varphi_{\RamiA{h}}$ is an open embedding of $\bfN \times \RamiA{\uv}$ into $\mathbf{X}$. 
        \item The collection  $\{\varphi_{\RamiA{h}}\}_{\RamiA{h}\in P}$ forms an atlas of $\mathbf{X}$.
    \end{itemize}
\end{lemma}

}
\begin{definition} For $i=0,\ldots,n-2$, let $\mathbf{D}_i \subseteq \mathbf{X}$ be the variety of pairs $((F_i),A)$ such that $A F_{\Rami{n-1-i}} \subseteq F_{\Rami{n-1-i}}$.\index{$\bfD_i$}
\Rami{Let $\mathbf{D} := \bigcup_{i=1}^{n-\RamiA{2}} \mathbf{D}_i$}.\index{$\bfD$}
\end{definition} 
\Rami{The following lemma is straightforward\RamiA{:}}
\begin{lemma}\label{lem:SNC} 
In the local charts $\varphi_{\RamiA{h}}$, \RamiA{the divisor} $\mathbf{D}_i$ is \RamiA{given by} the zero locus of the map $(u,A)\mapsto A_{\RamiA{n-i,n-1-i}}$. In particular, the $\mathbf{D}_i$ are prime divisors and $\mathbf{D}$ is \Rami{an SNC divisor}.
\end{lemma}

\begin{lemma} \label{lem:modification} Let $\pi : \mathbf{X} \rightarrow \Rami{\ug}$ be the projection. \begin{enumerate}
\item 
\RamiA{$\pi$ is proper, and t}he restriction of $\pi$ to $\mathbf{X} \smallsetminus \mathbf{D}$ is an open embedding. In particular, $\pi$ is a modification.
\item $\divv(\pi ^*(g))=\sum\limits_{i=0}^{n-2} (\Rami{i+1})\mathbf{D}_i$ 
\item Let $\omega_{\ug}$ be the invertible top form on ${\ug}$. Then $$\divv(\pi ^*( \omega_{\ug} ))=\sum_{i=1}^{n-2} i\mathbf{D}_i.$$
\end{enumerate} 
\end{lemma} 

\begin{proof} 
$ $
\begin{enumerate}
\item \RamiA{
Since $\mathbf{X}$ is Zariski closed in $\Fl(V)_1 \times \ug$, the map $\pi$ is projective.
Let $\mathbf{U} \subseteq \ug$ be the locus of $A\in \ug$ for which the vectors ${\bf e}_1,A{\bf e}_1,\ldots,A^{n-1}{\bf e}_1$ are linearly independent. Then $\pi(\mathbf{X} \smallsetminus \bfD)=\mathbf{U}$, $\mathbf{U}$ is Zariski open, and the map 
\[
A \mapsto \big( 0 \subseteq \linspan({\bf e}_1) \subseteq \linspan({\bf e}_1,A{\bf e}_1) \subseteq \cdots \subseteq \linspan({\bf e}_1,\ldots,A^{n-1}{\bf e}_1)=V ,A\big)
\]
is the inverse of $\pi$.}
\item In local coordinates, $\pi ^*(g)$ is given \Eitan{on} \RamiA{$\bfN\times \uv$} by the function 
\begin{align*}    
(\RamiA{x},A)& \mapsto  \det({\bf e}_1,\RamiA{x}A\RamiA{x} ^{-1} {\bf e}_1,\ldots,\RamiA{x}A^{n-1}\RamiA{x} ^{-1}{\bf e}_1)\\&=\RamiA{\det(x{\bf e}_1,xA{\bf e}_1,\ldots,xA^{n-1}{\bf e}_1)=}\det({\bf e}_1,A{\bf e}_1,\ldots,A^{\RamiA{n-1}}{\bf e}_1).
\end{align*}
For $A\in \mathfrak{v}\RamiA{:=\uv(F)}$, the matrix $B(A):=({\bf e}_1,A{\bf e}_1,\ldots,A^{\RamiA{n-1}}{\bf e}_1)$ is upper triangular with $B_{1,1}=1$, $B_{2,2}=A_{2,1}$ and, by induction, \RamiA{$$B_{i,i}=\prod_{j=1}^{i-1}A_{j+1,j}.$$}
Thus, \RamiA{$$\pi ^*(g)=\prod_{i=1}^n B_{i,i}=\prod_{i=1}^{n-1} A_{i+1,i}^{n-i}$$} and the claim follows.
\item Let
\[
\mathfrak{b}=\left\{ T\in \RamiA{\fg} \mid T_{i,j}=0\text{ if $i>j$} \right\} \subseteq \fv,
\]
\[
\mathfrak{m}=\left\{ T\in \RamiA{\fg} \mid T_{i,j}=0\text{ if $i \leq j+1$} \right\},
\]
\[
\mathfrak{n}=\left\{ T\in \RamiA{\fg} \mid T_{i,j}=0\text{ if $i=1$ or $j \leq i$} \right\},
\]
and let
\[
\text{$S$ be the matrix given by }S_{i,j}= \begin{cases} 1 & i+1=j \\ 0 & \text{else} \end{cases} .
\]
Then $\RamiA{\fg}=\mathfrak{v} \oplus \mathfrak{m}$, $\RamiA{N:=\bfN(F)}=1+\mathfrak{n}$, and right multiplication by $S$ is a linear isomorphism between $\mathfrak{m}$ and $\mathfrak{n}$.

\RamiA{In order to calculate $\div(\pi^*(\omega_\ug))$\Eitan{,} we compute the determinant of the differential of $\pi$ in local coordinates.}
In the local chart $\varphi_{id}$, \RamiA{the map} $\pi$ \RamiA{becomes} the map $(\RamiA{x},A)\mapsto \RamiA{x}  A \RamiA{x}^{\RamiA{-1}}$. Identifying the tangent space \RamiA{at a  point $x\in N$} with $\mathfrak{n}$ \RamiA{via the left multiplication by $x$}, the differential of $\pi$ at $(\RamiA{x},A)$ is the map $\mathfrak{n} \oplus \mathfrak{v} \rightarrow \RamiA{\fg}$ sending \RamiA{$(\eps,\delta)$ to $${x}[\eps,A]{x^{-1}}+{x}\delta\RamiA{x^{-1}}.$$} 
\Eitan{As $\det(Ad(x):{\fg} \to {\fg}) \equiv 1$, we have:}
\RamiA{
$\det(d_{(x,A)}\pi)=\det(d_{(1,A)}\pi)$.}

With the decompositions $\fn\oplus \fv$ and $\fm\oplus \fv$, the differential $\RamiA{d_{(1,A)}} \pi$ is block upper triangular and the diagonal block corresponding to $\mathfrak{v}$ is the identity. Thus, $\det(\RamiA{d_{(x,A)}}\pi)$ is the determinant of the map $T_A:\mathfrak{m} \rightarrow \mathfrak{m}$ given by $$T_A(\RamiA{Y})=\pr_{\mathfrak{m}}([\RamiA{Y}S,A]),$$ where $\pr_{\mathfrak{m}}:\Rami{\ug} \rightarrow \mathfrak{m}$ is the projection.

To compute $\det(T_A)$ note first that $T_{A+B}=T_A+T_B$. Define a filtration $\mathfrak{m} =\mathfrak{m}_1 \supseteq \mathfrak{m}_2 \supseteq \cdots \supseteq \mathfrak{m}_{n-1}=0$ by \[
\mathfrak{m}_k:=\left\{ T\in \mathfrak{m} \mid T_{i,j}=0\text{ if $j<k$} \right\}.
\]
Let $1 \leq k \leq n-1$. Then \begin{enumerate}
\item If $A\in \mathfrak{b}$ then $T_A(\mathfrak{m}_k) \subseteq \mathfrak{m}_{k+1}$. 
\item If $l\neq k$ then $T_{E_{(\RamiA{l+1,l})}}(\mathfrak{m}_k) \subseteq \mathfrak{m}_{k+1}$, \RamiA{where} $E_{(i,j)}$ is the $(i,j)$-elementary matrix.
\item $T_{E_{(\RamiA{k+1,k})}}(\mathfrak{m}_k)=\mathfrak{m}_k$ and acts on $\mathfrak{m}_k/ \mathfrak{m}_{k+1}$ as the identity.
\end{enumerate} 
Write $A=B+\sum A_{\RamiA{k+1,k}}E_{(\RamiA{k+1,k})}$ where $B\in \fb$. Then $$T_A=T_B+\RamiA{\sum_{k=1}^{n-1}} A_{\RamiA{k+1,k}}T_{E_{(\RamiA{k+1,k})}}$$ is  block upper triangular with respect to the filtration $\fm_k$ and acts \RamiA{by} the scalar $A_{\RamiA{k+1,k}}$ on $\mathfrak{m}_k/ \mathfrak{m}_{k+1}$.

Thus, $$\RamiA{\det(d_{(x,A)}\pi)=}\det(T_A)=\RamiA{\prod_{k=1}^{n-2}} (A_{\RamiA{k+1,k}})^{\dim(\fm_k/\fm_{k+1})}=\RamiA{\prod_{k=1}^{n-2}}(A_{\RamiA{k+1,k}})^{n-1-k}.$$ A similar computation holds in all other local charts, and the claim follows.
\end{enumerate} 
\end{proof} 
\begin{proof}[Proof of \Cref{thm:logcanonical}] By \Cref{prop:Jacobian of p} and a change of coordinates, we can replace the function $\mathcal{J}$ with the function $g(A):=\det({\bf e}_1,A{\bf e}_1,\ldots,A^{n-1}{\bf e}_1)$.

Let $\omega_\mathfrak{g}$ be the standard top differential form on $\mathfrak{g}$. The measure $| \omega_\mathfrak{g} |$ induced by $\omega_\mathfrak{g}$ is a Haar measure on $\mathfrak{g}$. Let $X:=\mathbf{X}(F)$. By \Cref{lem:modification}, the map $\pi:X \rightarrow \mathfrak{g}$ is proper, onto, and is a bijection between two conull sets. Thus,
\[
\int_{\mathfrak{g}} |g(A)|^{-s} \rho(A) dA =\int_{\mathfrak{g}} |g(A)|^{-s} \rho(A) | \omega_\mathfrak{g} | =\int_{X} | \pi ^*(g)(x) |^{-s} \rho(\pi(x)) | \pi ^* \omega_\mathfrak{g} |.
\]
Since $\rho$ is compactly supported and $\pi$ is proper, it is enough to show that the measure $| \pi ^*(g)(x) |^{-s} | \pi ^* \omega_\mathfrak{g} |$ is locally finite. 

Let $\xi \in X$. Define $J:=\left\{ j \mid \xi \in \mathbf{D}_j(F) \right\}$. Since\Rami{, by  \Cref{lem:SNC},} the $\mathbf{D}_j$ have normal crossings, there is a neighborhood $U$ of $\xi$ in $X$ and local analytic coordinates $x_1,\ldots,x_m$ centered at $\xi$ such that, for $j\in J$, the divisor $\mathbf{D}_j$ is locally defined by $x_j=0$. By \Cref{lem:modification}, there are nowhere zero analytic functions $u,u':U \rightarrow F$ such that 
\[
\pi ^*(g)=u \cdot \prod_{j\in J}x_j^{j+1}, \quad \pi ^* (\omega_\mathfrak{g}) =u'\cdot \prod_{j\in J}x_j^j \cdot dx_1 \wedge \cdots \wedge dx_m.
\]
Shrinking $U$ and rescaling, we can assume that $(x_1,\ldots,x_m)$ is a bijection between $U$ and $O_F^m$. \Eitan{Let}
$C\in \mathbb{R}$ \Eitan{be} 
such that $|u ^{-1} |, |u'|<C$\Eitan{. T}hen
\[
\int_U | \pi ^*(g)(x) |^{-s} | \pi ^* \omega_\mathfrak{g} |=\int_{O^m} u^{-s}u' \prod_{j\in J}|x_j|^{-s(j+1)+j} dx_1 \cdots dx_m \leq 
\]
\[
\leq C^{-s+1} \prod_{j\in J} \left( \int_{O_F} |x|^{-s(j+1)+j} dx \right) < \infty,
\]
where the last inequality is because $-s(j+1)+j>-1$.
\end{proof} 

\section{Proof of the main result}\label{sec:Pf.alm.an.frs}

We now combine the results of the previous sections.

\begin{proof}[Proof of \Cref{thm:intro-main}]
Write $\g=Y\times V$, where $Y=\operatorname{Mat}_{n\times(n-1)}(F)$ parametrizes the first $n-1$ columns of a matrix and $V=F^n$ parametrizes the last column.
Let $i:Y\hookrightarrow \g$ be the embedding $i(y)=(y,0)$. Identify $V$ with $\fc$.

By  \Rami{\Cref{lem:last-column-affine}}, the map $p:\g\to\fc$ has the form
$$
p(y,v)=f(y)v+g(y),
$$
where $f:Y\to\operatorname{End}(\fc)$ and $g:Y\to\fc$ are regular maps.

By the definition of $\cD$ in \Cref{def:D and J} we have $f=i^*(\cD)$. Since $\cD$ does not depend on the last column, it follows that $\cD=\pr_Y^*(f)$, where $\pr_Y:\g=Y\times V\to Y$ is the projection. Clearly, $\cJ=\pr_Y^*(\det(f))$ is the function on $\g$ studied in \S \ref{sec:J}.
Let $s<1$ and let $\rho\in C_c^\infty(Y)$. Choose $\eta\in C_c^\infty(V)$ such that $\int_V \eta(v)\,dv=1$. Define $\widetilde\rho(y,v):=\rho(y)\eta(v)$. Then $\widetilde\rho\in C_c^\infty(\g)$ and 
by \Cref{thm:logcanonical}, the integral
$$
\int_{\g}
|\cJ(x)|^{-s}\widetilde\rho(x)\,dx
$$
is finite.

Using \RamiA{the equality} $\cJ=\pr_Y^*(\det(f))$, we obtain
$$
\int_Y
|\det(f(y))|^{-s}\rho(y)\,dy
=\int_{\g}
|\cJ(x)|^{-s}\widetilde\rho(x)\,dx < \infty
$$


Therefore the assumptions of \Cref{thm:Lq-criterion} are satisfied. Hence for every $\phi\in C_c^\infty(\g)$, and for any $s < 1$, the measure $p_*(\phi\,dx)$ has density in $L^{\frac1{1-s}}(\fc)$.

Since every finite $q\ge1$ can be written in the form $q=\frac1{1-s}$ with $s<1$, we conclude that
$$
\frac{p_*(\phi\,dx)}{\mu_{\fc}}
\in
\bigcap_{1\le q<\infty}L^q(\fc).
$$

As every smooth compactly supported measure on $\g$ is of the form $\phi\,dx$ for some $\phi\in C_c^\infty(\g)$, the proof of \Cref{thm:intro-main} is complete.
\end{proof}
\begingroup
  \let\clearpage\relax
  \let\cleardoublepage\relax 
  \printindex
\endgroup
\bibliographystyle{alpha}
\bibliography{Ramibib}
\end{document}